\documentclass[reqno]{amsart}
\usepackage{amssymb}
\usepackage{amsfonts}

\allowdisplaybreaks
\newtheorem{theorem}{Theorem}[section]

\newtheorem{lemma}[theorem]{Lemma}

\theoremstyle{definition}

\theoremstyle{remark}

\theoremstyle{example}

\theoremstyle{fact}

\numberwithin{equation}{section}

\begin{document}
\title[Inverse scattering problems]{Inverse scattering problems where the potential is not absolutely continuous on the known interior
subinterval}
\author[Y. Guo]{Yongxia Guo}
\address{School of Mathematics and Information Science, Shaanxi Normal
University, Xi'an 710062, PR China}
\email{hailang615@126.com}
\author[G. Wei]{Guangsheng Wei$^{*}$}
\thanks{$^{*}$ Corresponding author.}
\address{School of Mathematics and Information Science, Shaanxi Normal
University, Xi'an 710062, PR China}
\email{weimath@snnu.edu.cn}

\begin{abstract}
The inverse scattering problem for the Schr$\mathrm{\ddot{o}}$dinger
operators on the line is considered when the potential is real valued and
integrable and has a finite first moment. It is shown that the potential on
the line is uniquely determined by the left (or right) reflection
coefficient alone provided that the potential is known on a finite interval
and it is not absolutely continuous on this known interval.
\end{abstract}

\maketitle
\date{}

\section{Introduction}

We consider the inverse scattering problems for one dimensional Schr$\mathrm{\ddot{o}}$dinger operators on the real line and analyze the unique recovery
of their potentials with the information known on a finite interval $[a,b]$.
Let $H$ be the self-adjoint Schr$\mathrm{\ddot{o}}$dinger operator on $L^{2}(\mathbb{R})$
\begin{equation}
H=-\frac{d^{2}}{dx^{2}}+V(x),  \label{1.1}
\end{equation}%
where the potential $V$ is real valued and belongs to $L_{1}^{1}(\mathbb{R})$, the class of measurable functions on the real axis $\mathbb{R}$ such that
$\int_{-\infty }^{\infty }(1+|x|)\left\vert V(x)\right\vert \mathrm{d}x$ is
finite.

The main purpose of the present paper is to prove the following theorem,
which is associated with the unique determination of the potentials on the
whole line.

\begin{theorem}
\label{blem1.1} Let $V$ be a real-valued potential belonging to $L_{1}^{1}(\mathbb{R}).$ If $V$ is a priori known on a finite interval $[a,b]$ and
it is not absolutely continuous on $[a,b]$, then $V$ on the whole line is uniquely determined by
either the left refection coefficient $L(k)$ or the right refection
coefficient $R(k)$ for $k\in \mathbb{R}$.
\end{theorem}

There are many results (see \cite{akt4,akt3,akt,bra,dei,ges,gre,kli,nov,ram1,run,sac} and the references
therein) related to the inverse scattering problem for one-dimensional Schr$\mathrm{\ddot{o}}$dinger equations defined on the entire real
line $\mathbb{R}$ with incomplete scattering data. These results show that if the potential
is known on a half-line, then the norming constants and even bound state
energies are not needed to recover the potential uniquely (some of these
papers are limited to the case where $V$ is assumed to vanish on a
half-line). In 1994, Weder (cf., \cite[p.222]{akt}) raised a question of
whether one can uniquely reconstruct $V$ by using the mixed scattering data
consisting of the bound state energies, the reflection coefficient $L(k)$
(or $R(k)$) for $k\in \mathbb{R}$ and the knowledge of the potential on a
finite interval $[a,b]$, i.e., all the bound state norming constants are
missing. Aktosun and Weder \cite{akt2} analyzed this inverse problem when
only one norming constant is missing, and proved that the missing norming
constant in the data can cause at most a double nonuniqueness in the
recovery, for which they illustrated the nonuniqueness with some explicit
examples. This enlighten us that, when the potential is known a priori on a
finite interval, we need additional condition to obtain the uniqueness for
such type of inverse scattering problems. Our Theorem 1.1 here gives an
effective answer to the uniqueness problem.

The method we use is a generalization of that used by Wei and Xu \cite{wei},
for which the basic idea is to relate our data to the Marchenko integral
equations that both integral equations have generalized degeneracy (see \cite{lev1,pur}) in the case that the part associated with the continuous
spectrum being the same for two systems.

\section{Proof of Theorem 1.1}

Consider the radial Schr$\mathrm{\ddot{o}}$dinger equation
\begin{equation}
-y^{\prime \prime }(k,x)+V(x)y(k,x)=k^{2}y(k,x),\text{ \ \ \ \ }x\in \mathbb{R},  \label{2.1}
\end{equation}
where $k^{2}$ is energy, $x$ is the space coordinate and the prime denotes
the derivative with respect to $x.$ It is known \cite[pp. 284-286]{mar} that
the scattering states of (\ref{2.1}) correspond to its solutions behaving
like $e^{ikx}$ or $e^{-ikx}$\ as $x\rightarrow \pm \infty.$ Such solutions
are the Jost solution from the left $f_{l}(k,x)$ and the Jost solution from
the right $f_{r}(k,x)$ satisfying
\begin{equation}
f_{l}(k,x)=\left\{
\begin{split}
& e^{ikx}+o(1),\ \ &&x\rightarrow +\infty , \\
& \frac{e^{ikx}}{T(k)}+\frac{L(k)}{T(k)}e^{-ikx}+o(1),\ \ &&x\rightarrow
-\infty ;
\end{split}
\right.  \label{2.2}
\end{equation}
\begin{equation}
f_{r}(k,x)=\left\{
\begin{split}
& \frac{e^{-ikx}}{T(k)}+\frac{R(k)}{T(k)}e^{ikx}+o(1),\ \ &&x\rightarrow
+\infty , \\
& e^{-ikx}+o(1),\ \ &&x\rightarrow -\infty .
\end{split}
\right.  \label{2.3}
\end{equation}
Here $T$ is the transmission coefficient, and $L$ and $R$ are the reflection
coefficients from the left and right, respectively. The bound states
correspond to the square-integrable solution of (\ref{2.1}), and such states
occur only at certain values $k=i\kappa _{j}$ on $\mathbb{I}^{+\text{{}}}:=i(0,+\infty )$ for $j=1,\cdots ,N$, which are exactly the poles
of $T(k)$. The so-called scattering data consists of
\begin{equation}
\left\{ L(k),\text{ }k\in \mathbf{\mathbb{R}}\right\} \cup \left\{ \kappa_{j},\text{ }m_{j}^{-}\right\} _{j=1}^{N}\text{ \ \ \ \ \textrm{or}\ \ \ }
\left\{ R(k),\text{ }k\in \mathbf{\mathbb{R}}\right\} \cup \left\{ \kappa_{j},\text{ }m_{j}^{+}\right\} _{j=1}^{N},\text{ }  \label{3.1}
\end{equation}%
where $m_{j}^{\pm }$ are the bound state norming constants corresponding to
the bound state energy $-\kappa _{j}^{2}$ defined as
\begin{equation}
m_{j}^{-}=||f_{r}(i\kappa _{j},\cdot )||^{-2},\text{ \ \ \ }
m_{j}^{+}=||f_{l}(i\kappa _{j},\cdot )||^{-2}.  \label{2.4}
\end{equation}%
It is well known (see, for example, \cite{dei,mar}) that the above
scattering data uniquely determines the potential $V$ on the whole line.

Before proving Theorem \ref{blem1.1} we shall first mention two lemmas which
will be needed later.

\begin{lemma}
\label{blem2.1} Let $y(k,x)$ be the nontrivial solution of the equation
\begin{equation}
-y^{\prime \prime }(k,x)+V(x)y(k,x)=k^{2}y(k,x),\ \ \ x\in [a,b],
\label{4.20}
\end{equation}
where $-\infty <a<b<+\infty .$ Then there exist a finite number of zeros of $y(k,x)$ on $[a,b]$, moreover these zeros are all simple.
\end{lemma}

\begin{proof}
The proof of the lemma is straightforward by \cite{weid}.
\end{proof}

When the parameter $k$ takes different finite values as $k=k_{s}$ for $s=1,\cdots ,n,$ it is easy to see that the number of all zeros of
$y(k_{s},x) $ on $[a,b]$ is also finite. This implies that there exists a
common point $x^{\prime }\in \lbrack a,b]$ such that $y(k_{s},x^{\prime})\neq 0$ for all $1\leq s\leq n.$

\begin{lemma}
\label{blem2.2} Let $\lambda _{1}<\lambda _{2}<\cdots <\lambda _{n}$ and
$\tilde{\lambda}_{1}<\tilde{\lambda}_{2}<\cdots <\tilde{\lambda}_{\tilde{n}}$ with $n\geq \tilde{n}$. Denote the $m\times n$ Vandermonde matrix
associated with entries $\{\lambda _{j}\}_{j=1}^{n}$ by
$V_{m\times n}\left[\lambda _{j}\right] _{j=1}^{n}$, that is,
\begin{equation}
V_{m\times n}\left[ \lambda _{j}\right] _{j=1}^{n}=
\begin{pmatrix}
1 & 1 & \cdots & 1 \\
\lambda _{1} & \lambda _{2} & \cdots & \lambda _{n} \\
\begin{array}{c}
\vdots%
\end{array}
& \vdots & \vdots & \vdots \\
\lambda _{1}^{m-1} & \lambda _{2}^{m-1} & \cdots & \lambda _{n}^{m-1}
\end{pmatrix}
.  \label{2.6}
\end{equation}
If there exists $m^{\prime }\leq \tilde{n}$ satisfying $\lambda _{j}=$\ $\tilde{\lambda}_{j}$ for $j=1,2,\cdots ,m^{\prime }$, and
$m:=n+\tilde{n}-m^{\prime }$,
\begin{equation}
V_{m\times n}[\lambda _{j}]_{j=1}^{n}A=V_{m\times \tilde{n}}[\tilde{\lambda}_{j}]_{j=1}^{\tilde{n}}\tilde{A},  \label{3.10}
\end{equation}%
where $A=[a_{1},\cdots ,a_{n}]^{T}\in \mathbb{R}^{n}$ and $\tilde{A}=[a_{1},\cdots ,\tilde{a}_{\tilde{n}}]^{T}\in \mathbb{R}^{\tilde{n}}$ are
such that $a_{j}\neq 0$\ and $\tilde{a}_{j}\neq 0$\ for all $1\leq j\leq \tilde{n}$,
then $\lambda _{j}=$\ $\tilde{\lambda}_{j}$, $a_{j}=$ $\tilde{a}_{j}$
for all $j=1,2,\cdots ,\tilde{n}$\ and $a_{j}=0$\ for $j=\tilde{n}+1,\cdots ,n$. In particular, in the case where $m^{\prime }=0$, the result
above still holds true.
\end{lemma}

\begin{proof}
The proof of the lemma is derived from \cite[Lemma 3.1]{wei}.
\end{proof}

For our purpose of this paper, together with the Schr$\mathrm{\ddot{o}}$dinger operator $H$ defined by (\ref{1.1}),
we consider another operator $\tilde{H}$ of the same form but with different coefficient $\tilde{V}$,
i.e., we consider another Schr$\mathrm{\ddot{o}}$dinger equation
\begin{equation}
-y^{\prime \prime }(k,x)+\tilde{V}(x)y(k,x)=k^{2}y(k,x),\text{ \ \ \ \ }x\in \mathbb{R}\text{.}  \label{2.5}
\end{equation}
We agree that, everywhere below if a symbol $\nu$ denotes an object related
to $H$, then $\tilde{\nu}$ will denote the analogous object related to $\tilde{H}$.

It is known that \cite[pp. 132-133]{lev} the Marchenko integral equation as
used in inverse scattering problems associated with the two operators $H$
and $\tilde{H}$ may be written as
\begin{equation}
B(x,y)+\Phi (x,y)+\int_{-\infty }^{x}B(x,t)\Phi (t,y)\mathrm{d}t=0,
\label{2.8}
\end{equation}
where\ $y<x$ and the function $\Phi (x,y)$\ has the following form
\begin{align}
\Phi (x,y)=& \frac{1}{2\pi }\int_{-\infty }^{\infty }[L(k)-\tilde{L}(k)] \tilde{f}_{r}(k,x)\tilde{f}_{r}(k,y)\mathrm{d}k  \notag  \label{2.9} \\
& +\sum\limits_{j=1}^{N}m_{j}^{-}\tilde{f}_{r}(i\kappa _{j},x)\tilde{f}_{r}(i\kappa _{j},y)-\sum\limits_{j=1}^{\tilde{N}}\tilde{m}_{j}^{-}\tilde{f}_{r}(i\tilde{\kappa}_{j},x)\tilde{f}_{r}(i\tilde{\kappa}_{j},y).
\end{align}%
Here $\tilde{f}_{r}(k,x)$ is the Jost solution of Eq. (\ref{2.5}) from the
left and $\tilde{m}_{j}^{-}$ is the Marchenko norming constant is similarly
defined by (\ref{2.4}) corresponding to the bound state $i\tilde{\kappa}_{j}$. Further, the function $B(x,y)$ satisfies the differential equation
\begin{equation}
\frac{\partial ^{2}B}{\partial x^{2}}-V(x)B=\frac{\partial ^{2}B}{\partial y^{2}}-\tilde{V}(y)B  \label{4.6}
\end{equation}
and condition
\begin{equation}
B(x,x)=\frac{1}{2}\int_{-\infty }^{x}[V(t)-\tilde{V}(t)]\mathrm{d}t.
\label{4.7}
\end{equation}
As a transformation operator, we have
\begin{equation}
f_{r}(k,x)=\tilde{f}_{r}(k,x)+\int_{-\infty }^{x}B(x,t)\tilde{f}_{r}(k,t) \mathrm{d}t.  \label{2.10}
\end{equation}
Similar results that related to the scattering data $\{R(k),\tilde{R}(k),k\in \mathbf{\mathbb{R}}\}\cup \{\kappa _{j},$
$m_{j}^{+}\}_{j=1}^{N}\cup \{\tilde{\kappa}_{j},$ $\tilde{m}_{j}^{+}\}_{j=1}^{\tilde{N}}$ are also
valid for the two operators $H$ and $\tilde{H}.$

By making use the Marchenko integral equation (\ref{2.8}), we are now in a
position to give the proof of Theorem \ref{blem1.1}.

\begin{proof}[Proof of Theorem \protect\ref{blem1.1}]
For the sake of simplicity, we shall only consider the uniqueness problem
for the left refection coefficient $L(k),$ the case for $R(k)$ can be
treated similarly. Consider two Schr$\mathrm{\ddot{o}}$dinger operators $H$
and $\tilde{H}.$ Under the hypothesis of Theorem \ref{blem1.1}, we have $L(k)=\tilde{L}(k)$ for $k\in \mathbb{R}$, $V(x)=\tilde{V}(x)$ a.e. on $[a,b]$, where two functions $V$ and $\tilde{V}$ are not absolutely continuous on $[a,b]$. Our purpose here is to prove $V=\tilde{V}$ a.e. on $\mathbb{R}$.

\textbf{Step 1.} We show that
\begin{equation}
\sum\limits_{j=1}^{N}(\kappa _{j}^{2})^{l}m_{j}^{-}(f_{r}\tilde{f}
_{r})(i\kappa _{j},x)=\sum\limits_{j=1}^{\tilde{N}}(\tilde{\kappa}
_{j}^{2})^{l}\tilde{m}_{j}^{-}(f_{r}\tilde{f}_{r})(i\tilde{\kappa}_{j},x)
\label{2.7}
\end{equation}
for $x\in [a,b]$ and $l=0,1,\cdots ,2M-1$ with $M=N+\tilde{N}.$

Since $L(k)=\tilde{L}(k)$ for $k\in \mathbb{R}$, it follows from (\ref{2.9})
that
\begin{equation}
\Phi (x,y)=\sum\limits_{j=1}^{N}m_{j}^{-}\tilde{f}_{r}(i\kappa _{j},x)\tilde{f}_{r}(i\kappa _{j},y)-
\sum\limits_{j=1}^{\tilde{N}}\tilde{m}_{j}^{-}\tilde{f}_{r}(i\tilde{\kappa}_{j},x)\tilde{f}_{r}(i\tilde{\kappa}_{j},y),
\label{2.13}
\end{equation}
which together with (\ref{2.8}) and (\ref{2.10}) yields
\begin{align}
B(x,y)& =-\Phi (x,y)-\int_{-\infty }^{x}B(x,t)\Phi (t,y)\mathrm{d}t  \notag
\\
& =\sum\limits_{j=1}^{\tilde{N}}\tilde{m}_{j}^{-}\tilde{f}_{r}(i\tilde{\kappa}_{j},x)\tilde{f}_{r}(i\tilde{\kappa}_{j},y)+
\sum\limits_{j=1}^{\tilde{N}}\tilde{m}_{j}^{-}\tilde{f}_{r}(i\tilde{\kappa}_{j},y)\int_{-\infty}^{x}B(x,t)\tilde{f}_{r}(i\tilde{\kappa}_{j},t)\mathrm{d}t  \notag \\
& \ \ -\sum\limits_{j=1}^{N}m_{j}^{-}\tilde{f}_{r}(i\kappa _{j},x)\tilde{f}_{r}(i\kappa _{j},y)-
\sum\limits_{j=1}^{N}m_{j}^{-}\tilde{f}_{r}(i\kappa_{j},y)\int_{-\infty }^{x}B(x,t)\tilde{f}_{r}(i\kappa _{j},t)\mathrm{d}t
\notag \\
& =\sum\limits_{j=1}^{\tilde{N}}\tilde{m}_{j}^{-}f_{r}(i\tilde{\kappa}_{j},x)\tilde{f}_{r}(i\tilde{\kappa}_{j},y)-
\sum\limits_{j=1}^{N}m_{j}^{-}f_{r}(i \kappa _{j},x)\tilde{f}_{r}(i\kappa _{j},y).  \label{2.14}
\end{align}%
It can be checked from \cite[Theorem 4.15(b)]{kir} that the solution $B(x,y)$ of the boundary value problem (\ref{4.6})-(\ref{4.7}) is a continuous function
on $\Omega =\{(x,y)\in \mathbb{R}^{2}:y\leq x\}.$ By (\ref{4.7}) and (\ref{2.14}) we have for $x\in \mathbb{R}$ that
\begin{equation}
\sum\limits_{j=1}^{\tilde{N}}\tilde{m}_{j}^{-}(f_{r}\tilde{f}_{r})(i\tilde{\kappa}_{j},x)-
\sum\limits_{j=1}^{N}m_{j}^{-}(f_{r}\tilde{f}_{r})(i\kappa_{j},x)=\frac{1}{2}\int_{-\infty }^{x}[V(t)-\tilde{V}(t)]\mathrm{d}t.
\label{4.9}
\end{equation}
Note that the condition $V(x)=\tilde{V}(x)$ a.e. for $x\in [a,b]$
yields
\begin{equation*}
\int_{-\infty }^{x}[\tilde{V}(t)-V(t)]\mathrm{d}t=\int_{-\infty }^{a}[\tilde{V}(t)-V(t)]\mathrm{d}t\text{ \ \ \ \textrm{for} }x\in [a,b].
\end{equation*}
This together with (\ref{4.9}) implies that for all $x\in [a,b]$
\begin{align}
& \sum\limits_{j=1}^{\tilde{N}}\tilde{m}_{j}^{-}(f_{r}\tilde{f}_{r})(i\tilde{\kappa}_{j},x)-
\sum\limits_{j=1}^{N}m_{j}^{-}(f_{r}\tilde{f}_{r})(i\kappa_{j},x)  \notag \\
=& \frac{1}{2}\int_{-\infty }^{a}[\tilde{V}(t)-V(t)]\mathrm{d}t  \notag \\
=& :C_{0}.  \label{2.15}
\end{align}
Differentiating Eq. (\ref{2.15}) with respect to $x,$ we infer for $x\in [a,b]$ that
\begin{equation}
\sum\limits_{j=1}^{\tilde{N}}\tilde{m}_{j}^{-}(f_{r}\tilde{f}_{r})^{\prime}(i\tilde{\kappa}_{j},x)-
\sum\limits_{j=1}^{N}m_{j}^{-}(f_{r}\tilde{f}_{r})^{\prime }(i\kappa _{j},x)=0.  \label{2.16}
\end{equation}
It should be noted that
\begin{align}
(f_{r}\tilde{f}_{r})^{\prime \prime }(k,x)=& (f_{r}^{\prime \prime }\tilde{f}_{r})(k,x)+(f_{r}\tilde{f}_{r}^{\prime \prime })(k,x)+2(f_{r}^{\prime }\tilde{f}_{r}^{\prime })(k,x)  \notag \\
=& (V(x)-k^{2})(f_{r}\tilde{f}_{r})(k,x)+(\tilde{V}(x)-k^{2})(f_{r}\tilde{f}_{r})(k,x)+2(f_{r}^{\prime }\tilde{f}_{r}^{\prime })(k,x)  \notag \\
=& 2(V(x)-k^{2})(f_{r}\tilde{f}_{r})(k,x)+2(f_{r}^{\prime }\tilde{f}_{r}^{\prime })(k,x)\text{ \ }\mathrm{a.e.}\text{ }\mathrm{on}\text{ }[a,b],
\text{\ }  \label{4.10}
\end{align}
where the last equation follows from the the condition $V(x)=\tilde{V}(x)$
a.e. on $[a,b].$ Differentiating Eq. (\ref{2.15}) with respect to $x$ for
twice, in other words, differentiating Eq. (\ref{2.16}) with respect to $x$,
we derive from (\ref{4.10}) that
\begin{align*}
& \sum\limits_{j=1}^{\tilde{N}}\tilde{m}_{j}^{-}[(V(x)+\tilde{\kappa}_{j}^{2})(f_{r}\tilde{f}_{r})(i\tilde{\kappa}_{j},x)+(f_{r}^{\prime }
\tilde{f}_{r}^{\prime })(i\tilde{\kappa}_{j},x)] \\
& \ \ \ \ \ \ \ \ -\sum\limits_{j=1}^{N}m_{j}^{-}[(V(x)+\kappa_{j}^{2})(f_{r}\tilde{f}_{r})(i\kappa _{j},x)+
(f_{r}^{\prime }\tilde{f}_{r}^{\prime })(i\kappa _{j},x)]=0\text{ \ \ }\mathrm{a.e.}\text{ }\mathrm{on}\text{ }[a,b].
\end{align*}
This together with (\ref{2.15}) gives that
\begin{align}
& \sum\limits_{j=1}^{\tilde{N}}\tilde{m}_{j}^{-}[\tilde{\kappa}_{j}^{2}(f_{r}\tilde{f}_{r})+(f_{r}^{\prime }\tilde{f}_{r}^{\prime })](i\tilde{\kappa}_{j},x)-
\sum\limits_{j=1}^{N}m_{j}^{-}[\kappa _{j}^{2}(f_{r}\tilde{f}_{r})+(f_{r}^{\prime }\tilde{f}_{r}^{\prime })](i\kappa _{j},x)  \notag \\
=& -V(x)\left[ \sum\limits_{j=1}^{\tilde{N}}\tilde{m}_{j}^{-}(f_{r}\tilde{f}_{r})(i\tilde{\kappa}_{j},x)-
\sum\limits_{j=1}^{N}m_{j}^{-}(f_{r}\tilde{f}_{r})(i\kappa _{j},x)\right]  \notag \\
=& -C_{0}V(x)\text{ \ \ }\mathrm{a.e.}\text{ }\mathrm{on}\text{ }[a,b].
\label{3.11}
\end{align}
On the one hand, the function of LHS of (\ref{3.11}) is an absolutely
continuous function on $[a,b],$ since the functions $f_{r}(k,x)$ and $\tilde{f}_{r}(k,x)$ are the solutions of (\ref{2.1}) and (\ref{2.5}), respectively.
On the other hand, the function $V(x)$ of RHS of (\ref{3.11}) is not absolutely continuous on $[a,b]$. Therefore, we infer that
\begin{equation}
C_{0}=0,  \label{4.3}
\end{equation}
and (\ref{2.15}) turns into
\begin{equation}
\sum\limits_{j=1}^{\tilde{N}}\tilde{m}_{j}^{-}(f_{r}\tilde{f}_{r})(i\tilde{\kappa}_{j},x)-\sum\limits_{j=1}^{N}m_{j}^{-}(f_{r}\tilde{f}_{r})(i\kappa_{j},x)=0.  \label{3.15}
\end{equation}

Furthermore, based on (\ref{4.3}), we have from (\ref{3.11}) that
\begin{equation}
\sum\limits_{j=1}^{\tilde{N}}\tilde{m}_{j}^{-}[\tilde{\kappa}_{j}^{2}(f_{r}\tilde{f}_{r})+(f_{r}^{\prime }\tilde{f}_{r}^{\prime })](i\tilde{\kappa}_{j},x)-
\sum\limits_{j=1}^{N}m_{j}^{-}[\kappa _{j}^{2}(f_{r}\tilde{f}_{r})+(f_{r}^{\prime }\tilde{f}_{r}^{\prime })](i\kappa _{j},x)=0.
\label{4.4}
\end{equation}
It should be noted that
\begin{align}
(f_{r}^{\prime }\tilde{f}_{r}^{\prime })^{\prime }(k,x)=& (f_{r}^{\prime
\prime }\tilde{f}_{r}^{\prime })(k,x)+(f_{r}^{\prime }\tilde{f}_{r}^{\prime
\prime })(k,x)  \notag \\
=& (V(x)-k^{2})(f_{r}\tilde{f}_{r}^{\prime })(k,x)+(\tilde{V}
(x)-k^{2})(f_{r}^{\prime }\tilde{f}_{r})(k,x)  \notag \\
=& (V(x)-k^{2})(f_{r}\tilde{f}_{r})^{\prime }(k,x),\text{ \ \ }\mathrm{a.e.}
\text{ }\mathrm{on}\text{ }[a,b],  \label{3.21}
\end{align}
Differentiating also Eq. (\ref{2.15}) with respect to $x$ for the third time
(i.e., differentiating Eq. (\ref{4.4}) with respect to $x$), we have from (\ref{3.21}) that
\begin{equation*}
\sum\limits_{j=1}^{\tilde{N}}\tilde{m}_{j}^{-}(V(x)+2\tilde{\kappa}_{j}^{2})(f_{r}\tilde{f}_{r})^{\prime }(i\tilde{\kappa}_{j},x)-
\sum \limits_{j=1}^{N}m_{j}^{-}(V(x)+2\kappa _{j}^{2})(f_{r}\tilde{f}_{r})^{\prime }(i\kappa _{j},x)=0.
\end{equation*}
This together with (\ref{2.16}) yields that
\begin{align}
& \sum\limits_{j=1}^{\tilde{N}}\tilde{m}_{j}^{-}\tilde{\kappa}_{j}^{2}(f_{r}
\tilde{f}_{r})^{\prime }(i\tilde{\kappa}_{j},x)-\sum
\limits_{j=1}^{N}m_{j}^{-}\kappa _{j}^{2}(f_{r}\tilde{f}_{r})^{\prime}(i\kappa _{j},x)  \notag \\
=& -\frac{V(x)}{2}\left[ \sum\limits_{j=1}^{\tilde{N}}\tilde{m}_{j}^{-}(f_{r}
\tilde{f}_{r})^{\prime }(i\tilde{\kappa}_{j},x)-\sum
\limits_{j=1}^{N}m_{j}^{-}(f_{r}\tilde{f}_{r})^{\prime }(i\kappa _{j},x) \right]  \notag \\
=& 0   \text{ \ \ }\mathrm{a.e.}\text{ }\mathrm{on}\text{ }[a,b].  \label{2.18}
\end{align}
Integrating Eq. (\ref{2.18}) from $a$ to $x$ with $x\in [a,b]$ gives
\begin{align}
& \sum\limits_{j=1}^{\tilde{N}}\tilde{m}_{j}^{-}\tilde{\kappa}_{j}^{2}(f_{r}
\tilde{f}_{r})(i\tilde{\kappa}_{j},x)-\sum\limits_{j=1}^{N}m_{j}^{-}\kappa_{j}^{2}(f_{r}\tilde{f}_{r})(i\kappa _{j},x)  \notag \\
=& \sum\limits_{j=1}^{\tilde{N}}\tilde{m}_{j}^{-}\tilde{\kappa}_{j}^{2}(f_{r}
\tilde{f}_{r})(i\tilde{\kappa}_{j},a)-\sum\limits_{j=1}^{N}m_{j}^{-}\kappa_{j}^{2}(f_{r}\tilde{f}_{r})(i\kappa _{j},a)  \notag \\
=& :C_{1}.  \label{4.1}
\end{align}%
Differentiating also Eq. (\ref{2.15}) with respect to $x$ for the fourth
time (i.e., differentiating Eq. (\ref{2.18}) with respect to $x$), we have
from (\ref{4.10}) and (\ref{4.1}) that
\begin{align}
& \sum\limits_{j=1}^{\tilde{N}}\tilde{m}_{j}^{-}\tilde{\kappa}_{j}^{2}[
\tilde{\kappa}_{j}^{2}(f_{r}\tilde{f}_{r})+(f_{r}^{\prime }\tilde{f}_{r}^{\prime })](i\tilde{\kappa}_{j},x)-
\sum\limits_{j=1}^{N}m_{j}^{-}\kappa_{j}^{2}[\kappa _{j}^{2}(f_{r}\tilde{f}_{r})+(f_{r}^{\prime }\tilde{f}_{r}^{\prime })](i\kappa _{j},x)  \notag \\
=&- V(x)\left[ \sum\limits_{j=1}^{\tilde{N}}\tilde{m}_{j}^{-}\tilde{\kappa}_{j}^{2}(f_{r}\tilde{f}_{r})(i\tilde{\kappa}_{j},x)-
\sum \limits_{j=1}^{N}m_{j}^{-}\kappa _{j}^{2}(f_{r}\tilde{f}_{r})(i\kappa _{j},x) \right]  \notag \\
=& -C_{1}V(x),\text{ \ \ }\mathrm{a.e.}\text{ }\mathrm{on}\text{ }[a,b].
\text{\ }  \label{4.14}
\end{align}%
Since the function $V(x)$ is not absolutely continuous on $[a,b]$, for the same reason of (\ref{3.11}), similar to (\ref{4.3}), we infer
\begin{equation}
C_{1}=0.  \label{4.13}
\end{equation}
Hence (\ref{4.1}) turns into
\begin{equation}
\sum\limits_{j=1}^{\tilde{N}}\tilde{m}_{j}^{-}\tilde{\kappa}_{j}^{2}(f_{r}
\tilde{f}_{r})(i\tilde{\kappa}_{j},x)-\sum\limits_{j=1}^{N}m_{j}^{-}\kappa_{j}^{2}(f_{r}\tilde{f}_{r})(i\kappa _{j},x)=0.  \label{3.14}
\end{equation}

Proceeding by induction, differentiating (\ref{2.15}) with respect to $x$
for $(2l+1)$ times, repeating the above proof for $l=0$ and $l=1$, and
making using of (\ref{4.10}) and (\ref{3.21}), analogous to (\ref{2.16}) and
(\ref{2.18}) we have for $x\in [a,b]$ that
\begin{equation*}
\sum\limits_{j=1}^{\tilde{N}}\tilde{m}_{j}^{-}(\tilde{\kappa}_{j}^{2})^{l}(f_{r}\tilde{f}_{r})^{\prime }(i\tilde{\kappa}_{j},x)-
\sum\limits_{j=1}^{N}m_{j}^{-}(\kappa _{j}^{2})^{l}(f_{r}\tilde{f}_{r})^{\prime }(i\kappa _{j},x)=0.
\end{equation*}
Integrating the above equation from $a$ to $x$ with $x\in [a,b],$
analogous to (\ref{2.15}) and (\ref{4.1}), we find
\begin{align}
& \sum\limits_{j=1}^{\tilde{N}}\tilde{m}_{j}^{-}(\tilde{\kappa}_{j}^{2})^{l}(f_{r}\tilde{f}_{r})(i\tilde{\kappa}_{j},x)-
\sum \limits_{j=1}^{N}m_{j}^{-}(\kappa _{j}^{2})^{l}(f_{r}\tilde{f}_{r})(i\kappa_{j},x)  \notag \\
=& \sum\limits_{j=1}^{\tilde{N}}\tilde{m}_{j}^{-}(\tilde{\kappa}_{j}^{2})^{l}(f_{r}\tilde{f}_{r})(i\tilde{\kappa}_{j},a)-
\sum \limits_{j=1}^{N}m_{j}^{-}(\kappa _{j}^{2})^{l}(f_{r}\tilde{f}_{r})(i\kappa_{j},a)  \notag \\
=& :C_{l}.  \label{4.12}
\end{align}
Differentiating also Eq. (\ref{2.15}) with respect to $x$ for $(2l+2)$
times, we have from (\ref{4.10}) and (\ref{4.12}) that
\begin{align*}
& \sum\limits_{j=1}^{\tilde{N}}\tilde{m}_{j}^{-}(\tilde{\kappa}_{j}^{2})^{l}[
\tilde{\kappa}_{j}^{2}(f_{r}\tilde{f}_{r})+(f_{r}^{\prime }\tilde{f}_{r}^{\prime })](i\tilde{\kappa}_{j},x)-
\sum\limits_{j=1}^{N}m_{j}^{-}(\kappa _{j}^{2})^{l}[\kappa _{j}^{2}(f_{r}\tilde{f}_{r})+(f_{r}^{\prime }
\tilde{f}_{r}^{\prime })](i\kappa _{j},x) \\
=& -V(x)\left[ \sum\limits_{j=1}^{\tilde{N}}\tilde{m}_{j}^{-}(\tilde{\kappa}_{j}^{2})^{l}(f_{r}\tilde{f}_{r})(i\tilde{\kappa}_{j},x)-
\sum \limits_{j=1}^{N}m_{j}^{-}(\kappa _{j}^{2})^{l}(f_{r}\tilde{f}_{r})(i\kappa_{j},x)\right] \\
=& -C_{l}V(x)\text{ \ \ }\mathrm{a.e.}\text{ }\mathrm{on}\text{ }[a,b].\text{\ }
\end{align*}
Based on the fact that the function $V(x)$ is not absolutely continuous on $[a,b]$, for the same reason of (\ref{3.11}) and (\ref{4.14}), similar to (\ref{4.3}) and (\ref{4.13}), we infer $C_{l}=0$ for
$l=2,\cdots,2M-1.$ This together with (\ref{4.12}) yields that (\ref{2.7})
holds.

\textbf{Step 2. }We show that
\begin{equation}
N=\tilde{N}\text{ \ }\mathrm{and}\text{ \ \ }\kappa _{j}=\tilde{\kappa}_{j},
\text{ }m_{j}^{-}=\tilde{m}_{j}^{-}\text{ \ }\mathrm{for}\text{ \ \ \ }
j=1,\cdots ,N.  \label{4.18}
\end{equation}

Without loss of generality, we assume $N>\tilde{N}.$ Since $V(x)=\tilde{V}(x)
$ a.e. on $[a,b],$ $\tilde{f}_{r}(k,x)$ and $f_{r}(k,x)$ both are nontrivial
solutions of Eq. (\ref{4.20}). This together with Lemma \ref{blem2.1}
implies that there exists a common point $x^{\prime }\in (a,b)$ such that
\begin{equation}
(f_{r}\tilde{f}_{r})(i\kappa _{i},x^{\prime })\neq 0\text{ \ \ }\mathrm{for}
\text{ }\mathrm{all}\text{ \ }i=1,\cdots ,N,  \label{4.23}
\end{equation}%
and
\begin{equation*}
(f_{r}\tilde{f}_{r})(i\tilde{\kappa}_{j},x^{\prime })\neq 0\text{ \ \ }
\mathrm{for}\text{ }\mathrm{all}\text{ \ }j=1,\cdots ,\tilde{N}.
\end{equation*}%
Denote by $V_{(N+\tilde{N})\times N}[\kappa _{j}^{2}]_{j=1}^{N}$ the
Vandermonde matrix associated with $\{\kappa _{j}^{2}\}_{j=1}^{N}.$ Note
that the Jost solution $f_{r}(k,x)$ of Eq. (\ref{2.1}) satisfies the reality
conditions $\overline{f_{r}(k,x)}=f_{r}(-\overline{k},x)$ for\ $\mathrm{Im}k\geq 0$ (see, for example, \cite[p.130]{dei}),
this gives that for all $k=i\kappa_{j}$ and $k=i\tilde{\kappa}_{j},$ the functions $f_{r}(k,x)$ and $\tilde{f}_{r}(k,x)$ both are real-valued. Denote the vector $A=(a_{1},\cdots
,a_{N})^{T}\in \mathbb{R}^{N}$ with
\begin{equation*}
a_{j}=m_{j}^{-}(f_{r}\tilde{f}_{r})(i\kappa _{j},x^{\prime }).
\end{equation*}
Similar notations can also be introduced for $\{\tilde{\kappa}_{j}^{2}\}_{j=1}^{\tilde{N}}$ corresponding to Vandermonde matrix the
$V_{(N+\tilde{N})\times \tilde{N}}[\tilde{\kappa}_{j}^{2}]_{j=1}^{\tilde{N}}$ and $\tilde{A}=(\tilde{a}_{1},\cdots ,\tilde{a}_{\tilde{N}})^{T}$ $\in \mathbb{R}^{\tilde{N}}$ with
\begin{equation*}
\tilde{a}_{j}=\tilde{m}_{j}^{-}(f_{r}\tilde{f}_{r})(i\tilde{\kappa}_{j},x^{\prime }).
\end{equation*}
Then by (\ref{2.7}) and $M=N+\tilde{N}$ we have
\begin{equation}
V_{M\times N}[\kappa _{j}^{2}]_{j=1}^{N}A=V_{M\times \tilde{N}}[\tilde{\kappa}_{j}^{2}]_{j=1}^{\tilde{N}}\tilde{A}.  \label{4.21}
\end{equation}
Applying Lemma \ref{blem2.2} to (\ref{4.21}) with $\lambda_{j}=\kappa_{j}^{2},\ \tilde{\lambda}_{j}=\tilde{\kappa}_{j}^{2},\ n=N$, $\tilde{n}=\tilde{N}$, we easily conclude that
\begin{equation}
\kappa_{j}=\tilde{\kappa}_{j},\text{ \ \ \ }m_{j}^{-}(f_{r}\tilde{f}_{r})(i\kappa _{j},
x^{\prime })=\tilde{m}_{j}^{-}(f_{r}\tilde{f}_{r})(i \tilde{\kappa}_{j},x^{\prime }),\text{ \ \ \ \ }j=1,\cdots ,\tilde{N},
\label{2.35}
\end{equation}
and further
\begin{equation}
m_{j}^{-}(f_{r}\tilde{f}_{r})(i\kappa _{j},x^{\prime })=0\text{ \ }\mathrm{for}\text{ \ \ }j=\tilde{N}+1,\cdots ,N.  \label{4.22}
\end{equation}
Thus a contradiction follows from (\ref{4.23}) and (\ref{4.22}). Therefore $N=\tilde{N},$ and (\ref{2.35}) further implies that
$\kappa_{j}=\tilde{\kappa}_{j}$ and $m_{j}^{-}=\tilde{m}_{j}^{-}$ for $j=1,\cdots ,N.$

Once we obtain (\ref{4.18}), by Marchenko's uniqueness theorem \cite{mar} we
have $V=\tilde{V}$\ a.e. on $\mathbb{R}.$ The proof is complete.
\end{proof}

\bigbreak

\textbf{Acknowledgement.} The authors would like to thank the referee for careful
reading of the manuscript and helping us to improve the presentation by providing
valuable and insightful comments. The research was supported in part by the NNSF
(11571212, 11601299), the Major Programs of National Natural Science Foundation of China (41390454), 
the China Postdoctoral Science Foundation
(2016M600760) and the Fundamental Research Funds for the Central
Universities (GK 201603007).

\end{document}